\documentclass[10pt,a4paper,twoside]{article}
\usepackage{amssymb}
\usepackage{amsmath}
\usepackage{fancyhdr}
\usepackage{amsthm}
\usepackage{rotate}
\usepackage{graphicx}
\usepackage[T1]{fontenc}

\usepackage{afterpage}  %
\theoremstyle{plain}
\newtheorem{theorem}{Theorem}[section]

\newtheorem{proposition}[theorem]{Proposition}

\newtheorem{corollary}[theorem]{Corollary}
\theoremstyle{definition}

\newtheorem{remark}[theorem]{Remark}

\begin{document}

\afterpage{\rhead[]{\thepage} \chead[\small W.A. Dudek and R.S.
Gigo\'n]{\small  Congruences on completely inverse
$AG^{**}$-groupoids} \lhead[\thepage]{} }
\begin{center}
\vspace*{2pt}
{\Large \textbf{Congruences on completely inverse $AG^{**}$-groupoids}}\\[30pt]
{\large \textsf{\emph{Wieslaw A. Dudek \ and \ Roman S. Gigo\'n
}}}
\\[30pt]
\end{center}
{\footnotesize\textbf{Abstract.} By a completely inverse $AG^{**}$-groupoid we mean an inverse $AG^{**}$-groupoid $A$ satisfying the identity $xx^{-1}=x^{-1}x$, where $x^{-1}$ denotes a unique element of $A$ such that $x=(xx^{-1})x$ and $x^{-1}=(x^{-1}x)x^{-1}.$ We show that the set of all idempotents of such groupoid forms a semilattice and the Green's relations $\mathcal{H,L, R,D}$ and $\mathcal{J}$ coincide on $A$.~The main result of this note says that any completely inverse $AG^{**}$-groupoid meets the famous Lallement's Lemma for regular semigroups. Finally, we show that the Green's relation $\mathcal{H}$ is both the least semilattice congruence and the maximum idempotent-separating congruence on any completely inverse $AG^{**}$-groupoid.}
 \footnote{
\hspace*{5mm} Green's relation.}

\section*{\centerline{1. Preliminaries}}\setcounter{section}{1}\setcounter{theorem}{0}

By an \emph{Abel-Grassmann's groupoid} (briefly an \emph{$AG$-groupoid}) we shall mean any groupoid which satisfies the identity
\begin{equation}\label{e1}
xy\cdot z=zy\cdot x.
\end{equation}
Such groupoid is also called a \emph{left almost semigroup} (briefly an \emph{$LA$-semi\-group}) or a \emph{left invertive groupoid} (cf.\,\cite{Hol}, \cite{KN} or \cite{MIq}).~This structure is closely related to a commutative semigroup, because if an $AG$-groupoid contains a right identity, then it becomes a commutative monoid.~Moreover, if an $AG$-groupoid $A$ with a left zero $z$ is finite, then (under certain conditions) $A\setminus\{z\}$ is a commutative group (cf.\,\cite{MK}).

\smallskip

One can easily check that in an arbitrary $AG$-groupoid $A$, the so-called \emph{medial law} is valid, that is, the equality
\begin{eqnarray}\label{medial}
ab\cdot cd=ac\cdot bd
\end{eqnarray}
holds for all $a,b,c,d\in A$.

Recall from \cite{PS} that an \emph{$AG$-band} $A$ is an $AG$-groupoid satisfying the identity $x^2=x$.~If in addition, $ab=ba$ for all $a,b\in A$, then $A$ is called an \emph{$AG$-semilattice}.

Let $A$ be an $AG$-groupoid and $B\subseteq A$.~Denote the set of all idempotents of $B$ by $E_B$, that is, $E_B=\{b\in B:b^2=b\}$.~From $(\ref{medial})$ follows that if $E_A \neq\emptyset$, then $E_A E_A\subseteq E_A$, therefore, $E_A$ is an $AG$-band.

\smallskip

Further, an $AG$-groupoid satisfying the identity
\begin{equation}\label{**}
x\cdot yz=y\cdot xz
\end{equation}
is said to be an \emph{$AG^{**}$-groupoid}.~Every $AG^{**}$-groupoid is {\em paramedial} (cf.\,\cite{BPS}), i.e., it satisfies the identity
\begin{eqnarray}\label{paramedial}
ab\cdot cd=db\cdot ca .
\end{eqnarray}
Notice that each $AG$-groupoid with a left identity is an $AG^{**}$-groupoid (see \cite{BPS}, too).~Furthermore, observe that if $A$ is an $AG^{**}$-groupoid, then (\ref{paramedial}) implies that if $E_A\neq\emptyset$, then it is an $AG$-semilattice.~Indeed, in this case $E_A$ is an $AG$-band and $ef=ee\cdot ff=fe\cdot fe=fe$ for all $e,f\in E_A$. Moreover, for $a,b\in A$ and $e\in E_A$, we have
$$
e\cdot ab=ee\cdot ab=ea\cdot eb=e(ea\cdot b)=e(ba\cdot e)=ba\cdot ee=ba\cdot e=ea\cdot b,
$$
that is,
\begin{equation}\label{e*}
e\cdot ab=ea\cdot b
\end{equation}
for all $a,b\in A$ and $e\in E_A$.~Thus, as a consequence, we obtain

\begin{proposition}\label{semilattice} The set of all idempotents of an $AG^{**}$-groupoid is either empty or a semilattice.
\end{proposition}

We say that an $AG$-groupoid $A$ with a left identity $e$ is an \emph{$AG$-group} if each of its elements has a \emph{left inverse} $a'$, that is, for every $a\in A$ there exists $a'\in A$ such that $a'a=e$. It is not difficult to see that such element $a'$ is uniquely determined and $aa'=e$. Therefore an $AG$-group has exactly one idempotent.

\smallskip

Let $A$ be an arbitrary groupoid, $a\in A$. Denote by $V(a)$ the set of all \emph{inverses} of $a$, that is, $$
V(a)=\{a^{*}\in A:a=aa^{*}\cdot a,~a^{*}=a^{*}a\cdot a^{*}\}.
$$
An $AG$-groupoid $A$ is called \emph{regular} (in \cite{BPS} it is called \emph{inverse}) if $V(a)\neq\emptyset$ for all $a\in A$. Note that $AG$-groups are of course regular $AG$-groupoids, but the class of all regular $AG$-groupoids is vastly more extensive than the class of all $AG$-groups. For example, every $AG$-band $A$ is regular, since $a=aa\cdot a\,$ for all $a\in A$. In \cite{BPS} it has been proved that in any regular $AG^{**}$-groupoid $A$ we have $|V(a)|=1$ $(a\in A)$, so we call it an \emph{inverse $AG^{**}$-groupoid}.~In this case, we denote a unique inverse of $a\in A$ by $a^{-1}$.~Notice that $(ab)^{-1}=a^{-1}b^{-1}$ for all $a,b\in A$.~Further, one can prove that in an inverse $AG^{**}$-groupoid $A$, we have $aa^{-1}=a^{-1}a$ if and only if $aa^{-1},a^{-1}a\in E_A$ (cf.\,\cite{BPS}).

\smallskip

Many authors studied various congruences on some special classes of
$AG^{**}$-groupoids and described the corresponding quotient
algebras as semilattices of some subgroupoids (see for example
\cite{BPS, MIq, MK1, MK2, P, PB}).~Also, in \cite{BPS, P} the authors
studied congruences on inverse $AG^{**}$-groupoids satisfying the
identity  $xx^{-1}=x^{-1}x$.~We will be called such groupoids
\emph{completely inverse $AG^{**}$-groupoids}.~A simple example of
such $AG^{**}$-groupoid is an $AG$-group.~In the light of
Proposition \ref{semilattice}, the set of all idempotents of any
completely inverse $AG^{**}$-groupoid forms a semilattice.

\smallskip

A nonempty subset $B$ of a groupoid $A$ is called a \emph{left ideal} of $A$ if $AB\subseteq B$. The notion of a \emph{right ideal} is defined dually. Also, $B$ is said to be an \emph{ideal} of $A$ if it is both a left and right ideal of $A$. It is clear that for every $a\in A$ there exists the least left ideal of $A$ containing the element $a$. Denote it by $L(a)$. Dually, $R(a)$ is the least right ideal of $A$ containing the element $a$.~Finally, $J(a)$ denotes the least ideal of $A$ containing $a\in A$.

\smallskip

In a similar way as in semigroup theory we define the \emph{Green's equi\-valences} on an $AG$-groupoid $A$ by putting:
\[
\begin{array}{c}
a\,\mathcal{L}\,b\iff L(a)=L(b),\\[4pt]
a\,\mathcal{R}\,b\iff R(a)=R(b),\\[4pt]
a\,\mathcal{J}\,b\iff J(a)=J(b),\\[4pt]
\mathcal{H}=\mathcal{L}\cap\mathcal{R},~~\mathcal{D}=\mathcal{L}\vee\mathcal{R}.
\end{array}
\]

\section*{\centerline{2. The main results}}\setcounter{section}{2}\setcounter{theorem}{0}

Let $A$ be a completely inverse $AG^{**}$-groupoid. Then
$$
a=(aa^{-1})a\in Aa
$$
for every $a\in A$.

\begin{proposition}\label{ideals} Let $A$ be a completely inverse $AG^{**}$-groupoid, $a\in A$.~Then$:$

$(a)$ \ $aA=Aa;$

$(b)$ \ $aA=L(a)=R(a)=J(a);$

$(c)$ \ $\mathcal{H}=\mathcal{L}=\mathcal{R}=\mathcal{D}=\mathcal{J};$

$(d)$ \ $aA=(aa^{-1})A;$

$(e)$ \ $aA=a^{-1}A;$

$(f)$ \ $eA=fA$ implies $\,e=f\,$ for all $\,e,f\in E_A.$
\end{proposition}

\begin{proof} $(a)$. Let $b\in A$. Then
$$
ab=(aa^{-1})a\cdot b=ba\cdot aa^{-1}=ba\cdot a^{-1}a=aa\cdot a^{-1}b=(a^{-1}b\cdot a)a\in Aa.$$
Thus $aA\subseteq Aa$. Also,
$$
ba=b\cdot (aa^{-1})a=aa^{-1}\cdot ba=ab\cdot a^{-1}a=ab\cdot aa^{-1}=a(ab\cdot a^{-1})\in aA,
$$
so $Aa\subseteq aA$.~Consequently, $aA=Aa$.

\smallskip
$(b)$. Obviously, it is sufficient to show that $aA=Aa$ is an ideal of $A$. Let $x=ab\in aA$ and $c\in A$.~Then we have $cx=c(ab)=a(cb)\in aA$ and $xc=(ab)c=(cb)a\in Aa=aA$.

\smallskip
$(c)$. It follows from $(b)$.

\smallskip
$(d)$. Let $b\in A$. Then $ab=(aa^{-1})a\cdot b=ba\cdot aa^{-1}\in A(aa^{-1})=(aa^{-1})A,$ that is, $aA\subseteq (aa^{-1})A$. Furthermore, $(aa^{-1})b=(ba^{-1})a\in Aa=aA$. Thus $(aa^{-1})A\subseteq aA$. Consequently, the condition $(d)$ holds.

\smallskip
$(e)$. By $(d)$, $aA=(aa^{-1})A=(a^{-1}a)A=(a^{-1}(a^{-1})^{-1})A$ $=a^{-1}A$.

\smallskip
$(f)$. Let $e,f\in E_A$ and $eA=fA$. Then $e\in fA$, that is, $e=fa$ for some $a\in A$.~Hence $fe=f(fa)=(ff)a$ (by Proposition \ref{semilattice}), and so $fe=e$. Similarly, $ef=f$. Since $E_A$ is a semilattice, $e=f$.
\end{proof}

\begin{corollary} Let $A$ be a completely inverse $AG^{**}$-groupoid.~Then each left ideal of $A$ is also a right ideal of $A$, and vice versa.~In particular, $$L\cap R=LR$$ for every $($left$)$ ideal $L$ and every $($right$)$ ideal $R$.
\end{corollary}

\begin{proof} Let $L$ be a left ideal of $A$ and $l\in L$.~Then $lA=Al\subseteq L$.~It follows that
$$L=\bigcup\{lA:l\in L\}.$$
Since each component $lA$ of the above set-theoretic union is a right ideal of $A$, then $L$ is itself a right ideal of $A$.~Similar arguments show that every right ideal of $A$ is a left ideal.

Clearly, $LR\subseteq L\cap R$.~Conversely, if $a\in L\cap R,$ then $a=(aa^{-1})a\in LR$. Hence $L\cap R=LR$.
\end{proof}

Let $A$ be a completely inverse $AG^{**}$-groupoid. Denote by $\mathcal{H}_a$ the equi\-valence $\mathcal{H}$-class containing the element $a\in A$. We say that $\mathcal{H}_a\leq\mathcal{H}_b$ if and only if $aA\subseteq bA$.

\smallskip

The following theorem is the main result of this paper.

\begin{theorem}\label{i-s} If $\rho$ is a congruence on a completely inverse $AG^{**}$-groupoid \nolinebreak $A$ and $\,a\rho\in E_{A/\rho}$ $(a\in A)$, then there exists $\,e\in E_{a\rho}$ such that $\mathcal{H}_e\leq\mathcal{H}_a$.
\end{theorem}

\begin{proof} Let $\rho$ be a congruence on $A,$ $a\in A$ and $a\rho a^2$. We know that there exists $x\in A$ such that $a^2=a^2x\cdot a^2,$ $\,x=xa^2\cdot x$ and $a^2x=xa^2\in E_A$. Notice that
$$
a^2x\cdot aa=a(a^2x\cdot a)=a(xa^2\cdot a)=a(aa^2\cdot x)=aa^2\cdot ax=a^2\cdot a^2x=a^2\cdot xa^2,
$$
i.e., $a^2=a^2\cdot xa^2$. Put $e=a\cdot xa$. Then $e\,\rho\,(a^2\cdot xa^2)=a^2\,\rho\,a$. Hence $e\in a\rho$. Also,
$$
e^2=(a\cdot xa)(a\cdot xa)=a((a\cdot xa)\cdot xa)=a(ax\cdot(xa\cdot a))=a(ax\cdot a^2x).
$$
Further,
$$
ax\cdot a^2x=ax\cdot xa^2=a^2x\cdot xa=xa^2\cdot xa=(xa^2\cdot x)a
$$
by \eqref{e*}, since $xa^2\in E_A$. Hence $ax\cdot a^2x=xa$. Consequently,
$$
e^2=a\cdot xa=e\in E_A.
$$
Thus, $e\in E_{a\rho}$.

Finally, let $b\in A$. Then $eb=(a\cdot xa)b =(b\cdot xa)a\in Aa=aA$, therefore, $eA\subseteq aA$, so $\mathcal{H}_e\leq\mathcal{H}_a$.
\end{proof}

We say that a congruence $\rho$ on a groupoid $A$ is \emph{idempotent-separating} if $e\rho f$ implies that $\,e=f\,$ for all $e,f\in E_A$. Furthermore, $\rho$ is a \emph{semilattice} congruence if $A/\rho$ is a semilattice.~Finally, $A$ is said to be a \emph{semilattice $A/\rho$ of $AG$-groups} if $\rho$ is a semilattice congruence and every $\rho$-class of $A$ is an $AG$-group.

\begin{corollary}\label{H} Let $A$ be a completely inverse $AG^{**}$-groupoid.~Then$:$

\smallskip
$(a)$ \ $\mathcal{H}$ is the least semilattice congruence on $A;$

\smallskip
$(b)$ \ $\mathcal{H}$ is the maximum idempotent-separating congruence on $A;$

\smallskip
$(c)$ \ $A$ is a semilattice $A/\mathcal{H}\cong E_A$ of $AG$-groups $\,\mathcal{H}_e$ $(e\in E_A)$.
\end{corollary}

\begin{proof} $(a)$. Let $aA=bA$ and $c,x\in A$.~Then $x\cdot ca=c\cdot xa$.~On the other hand,
$$
xa\in Aa=aA=bA=Ab,
$$
i.e., $xa=yb$, where $b\in A$, so $x\cdot ca=c\cdot yb=y\cdot cb\in A(cb)$.~Thus $A(ca)\subseteq A(cb)$. By symmetry, we conclude that $A(ca)=A(cb)$. Moreover, $a=yb$ for some $y\in A$. Hence $ac\cdot x=xc\cdot a=xc\cdot yb=bc\cdot yx\in (bc)A$.~Thus $(ac)A\subseteq (bc)A$.~In a similar way we can obtain the converse inclusion, so $(ac)A=(bc)A$.~Consequently, $\mathcal{H}$ is a congruence (by Proposition \ref{ideals} $(b)$). In the light of Proposition \ref{ideals} $(d)$, every $\mathcal{H}$-class contains an idempotent of $A$.~This implies that $A/\mathcal{H}$ is a semilattice, that is, $\mathcal{H}$ is a semilattice congruence on $A$.

Suppose that there is a semilattice congruence $\rho$ on $A$ such that $\mathcal{H}\nsubseteq\rho.$\\ Then the relation $\mathcal{H}\cap\rho$ is a semilattice congruence which is properly contained in $\mathcal{H}$, and so not every $(\mathcal{H}\cap\rho)$-class contains an idempotent of $A$, since each $\mathcal{H}$-class contains exactly one idempotent (Proposition \ref{ideals} $(f)$), a contradiction with Theorem \ref{i-s}.~Consequently, $\mathcal{H}$ must be the least semilattice congruence on $A$.

\smallskip

$(b)$. By $(a)$ and Proposition \ref{ideals} $(f)$, $\mathcal{H}$ is an idempotent-separating congruence on $A$.~On the other hand, if $\rho$ is an idempotent-separating congruence on $A$ and $(a,b)\in\rho$, then $(a^{-1},b^{-1})\in\rho$, so $(aa^{-1},bb^{-1})\in\rho$. Hence $aa^{-1}=bb^{-1}$. Let $x\in A$. Then
$$
xa=x(aa^{-1}\cdot a)=x(bb^{-1}\cdot a)=bb^{-1}\cdot xa=(xa\cdot b^{-1})b\in Ab.$$
Thus $Aa\subseteq Ab$.~By symmetry, we conclude that $Aa=Ab$.~Consequently, $a\,\mathcal{H}\,b$ (Proposition \ref{ideals} $(b)$), that is, $\rho\subseteq\mathcal{H}$, as required.

\smallskip $(c)$.
We show that every $\mathcal{H}$-class of $A$ is an $AG$-group.~In view of the above and Proposition \ref{ideals} $(d), (e)$, each $\mathcal{H}$-class is an $AG^{**}$-groupoid.~Consider an arbitrary $\mathcal{H}$-class $\mathcal{H}_e$ ($e\in E_A$).~Let $a\in\mathcal{H}_e$.~Then $aa^{-1}\in\mathcal{H}_e$.~Hence $aa^{-1}=e$ and so $ea=a$, that is, $e$ is a left identity of $\mathcal{H}_e$.~Since $a^{-1}a=e$ and $a^{-1}\in\mathcal{H}_e$, then $\mathcal{H}_e$ is an $AG$-group.~Obviously, $A/\mathcal{H}\cong E_A$.~Consequently, $A$ is a semilattice $A/\mathcal{H}\cong E_A$ of $AG$-groups $\mathcal{H}_e$ $(e\in E_A)$.
\end{proof}

We say that an ideal $K$ of a groupoid $A$ is the \emph{kernel} of $A$ if $K$ is contained in every ideal of $A$.~If in addition, $K$ is an $AG$-group, then it is called the \emph{$AG$-group kernel} of $A$.~Finally, a congruence $\rho$ on $A$ is said to be an \emph{$AG$-group} congruence if $A/\rho$ is an $AG$-group.

\begin{corollary} Let $A$ be a completely inverse $AG^{**}$-groupoid.~If $\,e$ is a zero of $E_A$, then $\mathcal{H}_e=eA$ is the $AG$-group kernel of $A$ and the map $\varphi:A\to eA$ given by $a\varphi=ea$ $(a\in A)$ is an epimorphism such that $x\varphi=x$ for all $x\in eA$.
\end{corollary}

\begin{proof} Obviously, $\mathcal{H}_e\subseteq eA$.~Conversely, if $x=ea\in eA$, then
$$
xx^{-1}=ea\cdot ea^{-1}=ee\cdot aa^{-1}=e.
$$
In a view of Proposition \ref{ideals} $(d)$, $x\in\mathcal{H}_e$.~Consequently, $\mathcal{H}_e=eA$.~If $I$ is an ideal of $A$, then clearly $E_I\neq\emptyset$.~Let $i\in E_I$. Then $e=ei\in E_I$.~Hence $a=ea\in I$ for all $a\in\mathcal{H}_e$, so $\mathcal{H}_e\subseteq I$.~Thus $\mathcal{H}_e=eA$ is the $AG$-group kernel of $A$.~Also, for all $a,b\in A$, $(a\varphi)(b\varphi)=(ea)(eb)=(ee)(ab)=e(ab)=(ab)\varphi$, i.e., $\varphi$ is a homomorphism of $A$ into $eA$.~Evidently, $\varphi$ is surjective. Finally, $\varphi_{|eA}=1_{eA}$ (by Proposition \ref{semilattice}).
\end{proof}

\begin{corollary} Let $A$ be a completely inverse $AG^{**}$-groupoid.~If $\,e$ is a zero of \nolinebreak $E_A$, then
$$\sigma=\{(a,b)\in A\times A:ea=eb\}$$
is the least $AG$-group congruence on $A$ and $A/\sigma\cong\mathcal{H}_e$.
\end{corollary}

\begin{proof}
It is clear that $\sigma$ is an $AG$-group congruence on $A$ induced by $\varphi$ (defined in the previous corollary). If $\rho$ is also an $AG$-group congruence on $A$ and $a\,\sigma\,b$, then $(e\rho)(a\rho)=(e\rho)(b\rho)$. By cancellation, $a\,\rho\,b$ and so $\sigma\subseteq\rho$. Obviously, $A/\sigma\cong\mathcal{H}_e$.
\end{proof}

\begin{remark} Let $I$ be an ideal of a completely inverse $AG^{**}$-groupoid \nolinebreak $A$. The relation $\rho_I=(I\times I)\cup 1_A$ is a congruence on $A$. If $e$ is a zero of $E_A$, then $\mathcal{H}_e$ is an ideal of $A$ and $\sigma\,\cap\,\rho_{\mathcal{H}_e}=1_A$.~It follows that $A$ is a subdirect product of the group $\mathcal{H}_e$ and the completely inverse $AG^{**}$-groupoid $A/\mathcal{H}_e$.~Note that we may think about $A/\mathcal{H}_e$ as a groupoid $B=(A\setminus\mathcal{H}_e)\cup\{e\}$ with zero $e$, where all products $ab\in\mathcal{H}_e$ are equal \nolinebreak $e$. In fact, $fg=e$ in $A$ ($f,g\in E_A$) if and only if $\mathcal{H}_f\mathcal{H}_g\subseteq\{e\}=\mathcal{H}_e$ in $B$.

Obviously, in any finite completely inverse $AG^{**}$-groupoid $A$, the semilattice $E_A$ has a zero.
\end{remark}

\medskip
\footnotesize{
\rightline{Received \ August 29, 2012}

Institute of Mathematics and Computer Science

         Wroclaw University of Technology

        Wyb. Wyspianskiego 27

         $50-370$ Wroclaw

         Poland

E-mails: wieslaw.dudek@pwr.wroc.pl, \ \ \ \
romekgigon@tlen.pl}

\end{document}